\documentclass[11pt]{amsart}

\usepackage{amssymb}

\DeclareMathOperator{\dom}{dom}
\DeclareMathOperator{\rge}{rge}
\DeclareMathOperator{\Ult}{Ult}
\DeclareMathOperator{\crit}{crit}

\newtheorem{question}{Question} 

\newtheorem{theorem}{Theorem} 
\newtheorem{lemma}{Lemma} 
\newtheorem{corollary}[lemma]{Corollary} 
\newtheorem{fact}[lemma]{Fact}

\theoremstyle{remark}
 
\newtheorem{remark}[lemma]{Remark} 

\title{Singular cardinals and strong extenders}

\author{Arthur W.~Apter}

\address{Department of Mathematics, Baruch College of CUNY, New York NY 10010
  \&  The CUNY Graduate Center, Mathematics, 365 Fifth Avenue, New York NY 10016, USA} 

\email{awapter@alum.mit.edu} 

\urladdr{http://faculty.baruch.cuny.edu/aapter}

\author{James Cummings}

\address{Department of Mathematical Sciences, Carnegie Mellon University,
      Pittsburgh PA 15213, USA}

\email{jcumming@andrew.cmu.edu}

\urladdr{http://www.math.cmu.edu/math/faculty/cummings.html} 

\author{Joel David Hamkins} 

\address{The City University of New York, 
        The College of Staten Island of CUNY, \&
        The CUNY Graduate Center, Mathematics,
        365 Fifth Avenue, New York, NY 10016 USA}

\email{jhamkins@gc.cuny.edu}

\urladdr{http://jdh.hamkins.org/}



\thanks{The first author was partially supported by PSC-CUNY grant 64227-00-42. The second author was 
   partially supported by NSF grant DMS-1101156.  The third author was partially supported by
  NSF grant DMS-0800762, PSC-CUNY grant 64732-00-42 and Simons Foundation grant 209252.}

\date{\today}

\begin{document}

\begin{abstract} 
     We investigate the circumstances under which there exist a singular cardinal $\mu$ and
    a short $(\kappa, \mu)$-extender $E$ witnessing ``$\kappa$ is $\mu$-strong'', such that
    $\mu$ is singular in $\Ult(V, E)$. 
\end{abstract}

\maketitle

\section{Introduction} \label{intro}

    In the course of some work on his doctoral dissertation \cite{Cody}, Brent Cody
    encountered some issues which caused him to raise  the following question:
\begin{question} \label{quest}
    If $\kappa$ is $\mu$-strong for some singular cardinal $\mu > \kappa$, is there
    a $(\kappa, \mu)$-extender $E$ which witnesses that $\kappa$ is $\mu$-strong and
    is such that
\[
    \Ult(V, E) \models \mbox{``$\mu$ is singular''?}
\]
\end{question}

   Before describing our results we make a few preliminary remarks. It is easy to see
   that if $\kappa$ is $\mu$-strong for some cardinal $\mu > \kappa$ with $\vert V_\mu \vert = \mu$,
   then there is a $(\kappa, \mu)$-extender witnessing that $\kappa$ is $\mu$-strong. The hypothesis
   that $\vert V_\mu \vert = \mu$ is necessary to obtain such an extender. For any $\mu > \kappa$ with
   $\vert V_\mu \vert = \mu$ and any embedding $j: V \longrightarrow M$ with $\crit(j) = \kappa$,
   $V_\mu \subseteq M$ if and only if $M$ contains all the bounded subsets of $\mu$. 

   Suppose now that $\kappa < \mu$ for some singular $\mu$ with $\vert V_\mu \vert = \mu$, and 
   $E$ is a $(\kappa, \mu)$-extender witnessing that $\kappa$ is $\mu$-strong. We will say that
   $E$ is a {\em good witness}  if $\mu$ is singular in $\Ult(V, E)$, and a {\em bad witness} if
   $\mu$ is regular (and hence inaccessible) in $\Ult(V, E)$.  

   The authors observed that if $\kappa$ is $\nu$-strong for some inaccessible $\nu > \kappa$, then 
   there is a club set  $C \subseteq \nu$ such that  for every singular $\mu \in C$
   there is a bad witness (see Fact \ref{bad} below). Brent Cody recently informed us that this result
   was already known, and had appeared in print in a paper of Friedman and Honzik \cite[Observation 2.8]{FrHon}.
   However this result does not completely settle Question \ref{quest}.

   Our main results are:

\begin{enumerate}

\item  If there is a bad witness, then there is a normal measure on $\kappa$ concentrating
     on $\alpha$ which are strong up to $\alpha^*$, where $\alpha^*$ is the least inaccessible  
     cardinal greater than $\alpha$.

\item  If $\kappa$ is $\nu$-strong for some cardinal $\nu$, then for every singular
     $\mu$ such that $\kappa < \mu < \nu$ and $\vert V_\mu \vert = \mu$  there is a good witness. 

\item  (From suitable large cardinal assumptions) 
\begin{enumerate}
\item 
 It is consistent that $\mu$ is singular, there is exactly one $(\kappa, \mu)$-extender $E$
     witnessing that $\kappa$ is $\mu$-strong, and $E$ is a good witness.
\item 
 It is consistent that $\mu$ is singular, there is exactly one $(\kappa, \mu)$-extender $E$
     witnessing that $\kappa$ is $\mu$-strong, and $E$ is a bad witness.
\end{enumerate} 
\end{enumerate} 

    The last result uses models of the form $L[\vec E]$ where $\vec E$ is a coherent
    sequence of non-overlapping extenders. 

\section{Proofs of the main results}

\subsection{Bad witnesses} 

   We begin with an easy reflection argument, giving a lower bound in consistency strength for the existence of a bad witness.
    
\begin{theorem} If there is a bad witness for ``$\kappa$  is $\mu$-strong'', then
    there is a normal measure on $\kappa$ concentrating on $\alpha$ which are strong
    up to $\alpha^*$, where $\alpha^*$ is the least inaccessible cardinal greater than
    $\alpha$.
\end{theorem} 

\begin{proof} Let $E$ be a bad witness, so that $\mu$ is inaccessible in
    $\Ult(V, E)$. Since $V_\mu \subseteq \Ult(V, E)$, $\Ult(V, E)$ contains 
    extenders witnessing that $\kappa$ is $\nu$-strong for every $\nu < \mu$,
    so
\[
    \Ult(V, E) \models \mbox{``$\mu$ is inaccessible and $\kappa$ is strong up to $\mu$''.}
\]
   If $U$ is the normal measure derived from $j_E$ then $U$ concentrates on the
   set of $\alpha$ which are strong up to $\alpha^*$. 
\end{proof} 

   The following fact, giving an upper bound for the existence of a bad witness,
   was observed by Friedman and Honzik \cite[Observation 2.8]{FrHon}. We give the proof
   because we need the idea later in Theorem \ref{badlewit}.

\begin{fact}  \label{bad}  If $\kappa$ is $\nu$-strong for some inaccessible $\nu > \kappa$, then 
    there is a club set  $C \subseteq \nu$ such that  for every singular $\mu \in C$
     there is a bad witness. 
\end{fact} 

\begin{proof}   Let $E$ be a  $(\kappa, \nu)$-extender witnessing that $\kappa$ is $\nu$-strong, and
   let $j_E: V \longrightarrow M_E = \Ult(V, E)$ be the corresponding ultrapower map.
   We note that since $\nu$ is inaccessible in $V$, $\nu$ is inaccessible in $\Ult(V, E)$.

 Recall that
\[
     M_E = \{ j_E(f)(a) : a \in [\nu]^{<\omega}, \dom(f) = [\kappa]^{<\omega} \}. 
\]
    For any $\mu < \nu$ we may form the $(\kappa, \mu)$-extender $E \restriction \mu$
    and the corresponding ultrapower map  $j_{E \restriction \mu} : V \longrightarrow M_{E \restriction \mu}$.
    It is easy to see that there is an elementary embedding $k_\mu$ from  $M_{E \restriction \mu}$ to
    $M_E$  such that $k_\mu \circ j_{E \restriction \mu} = j_E$. This embedding is given by 
    the formula $k_\mu: j_{E \restriction \mu} (f)(a) \mapsto j_E (f)(a)$ for $a \in [\mu]^{<\omega}$.
    Note that $M_E = \bigcup_{\mu < \nu} \rge(k_\mu)$, so that in particular
    $\nu \in \rge(k_\mu)$ for all large enough $\mu < \nu$.

    We now define a function $F$ with domain $\nu$, where 
\[
    F(\mu) = \max \{ \vert V_\mu \vert, \sup(\rge(k_\mu) \cap \nu) \}.
\]  
    Since $\nu$ is inaccessible and
\[
    \rge(k_\mu) \cap \nu  \subseteq \{ j_E(f)(a) : a \in [\mu]^{<\omega}, f:[\kappa]^{<\omega} \longrightarrow \kappa \},
\]
    $\rge(F) \subseteq \nu$. If $\mu$ is a closure point of $F$ with $\nu \in \rge(k_\mu)$, then  
    $\vert V_\mu \vert = \mu$ and 
\[
    \nu \cap \rge(k_\mu) = \mu. 
\]
    Let $C$ be the club set of closure points $\mu$ of $F$ such that $\nu \in \rge(k_\mu)$. 

    For each $\mu \in C$, we have that $k_\mu(\mu) = \nu$, so that
    by elementarity $\mu$ is inaccessible in $M_{E \restriction \mu} = \Ult(V, E \restriction \mu)$.
    It follows easily that for every singular $\mu \in C$ the extender $E \restriction \mu$ is a 
    bad witness.    
\end{proof}     

\subsection{Good witnesses} 

\begin{theorem} \label{goodwit} 
 If $\kappa$ is $\nu$-strong for some cardinal $\nu$, then for every singular
     $\mu$ such that $\kappa < \mu < \nu$ and $\vert V_\mu \vert = \mu$  there is a good witness. 
\end{theorem} 

\begin{proof} 
   Fix an extender $E$ witnessing that $\kappa$ is $\nu$-strong. We prove the claim by induction
   on $\mu$. Suppose that $\mu$ is a minimal counterexample. As above we may form
   $j_E: V \longrightarrow M_E$ and factor it through the ultrapower by $E \restriction \mu$,
   obtaining an embedding $k_\mu$ from $M_{E \restriction \mu}$ to $M_E$ such that
   $k_\mu \circ j_{E \restriction \mu} = j_E$. 

   As usual $\mu \subseteq \rge(k_\mu)$, and we claim that in this case also
   $\mu \in \rge(k_\mu)$. To see this we observe that $V_\nu \subseteq M_E$, 
   so that by a routine calculation 
\[
   M_E \models \mbox{``$\mu$ is the least cardinal with no good witness''}.
\]
   Hence $\mu$ is definable from $\kappa$ in $M_E$, which implies that
   $\mu \in \rge(k_\mu)$. 

   It follows that $k_\mu(\mu) = \mu$. Since $V_\nu \subseteq M_E$ we see that
   $\mu$ is singular in $M_E$, and hence by elementarity $\mu$ is singular
   in $M_{E \restriction \mu}  = \Ult(V, E \restriction \mu)$. 
   So $E \restriction \mu$ is a good witness for $\mu$, contradicting the choice
   of $\mu$ as the least counterexample. 

\end{proof} 

\subsection{Unique witnesses} 

    To prove the remaining results we need some analysis of extenders in models
    of the form $L[\vec E]$ where $\vec E$ is a  coherent sequence
    of non-overlapping extenders. We refer the reader to Mitchell's excellent survey paper \cite{Mitchell1}
    for a detailed account of these models; we adopt the terminology and conventions of that paper, in particular
    we note that  $E(\kappa, \beta)$ is a  total $(\kappa, \kappa + 1 + \beta)$-extender.

    The key fact is the {\em Comparison Lemma} \cite[Lemma 3.15]{Mitchell1}, which
    states that (under the right hypotheses) two models with extender sequences can be iterated so that 
    the images of the original extender sequences are ``lined up''. In our proofs we will freely use the Comparison Lemma
    and some immediate consequences.   
    In all comparison iterations which appear below,
    the extender sequences are coherent, so that the critical points are strictly increasing. 
   
 We will also use the following fairly standard facts about models of the form $L[\vec E]$,
    all of which are due to Mitchell  \cite{Mitchell2,Mitchell3}. For the convenience of the reader,
   we have included references or sketchy proofs. Some of the arguments would be slicker with 
   an appeal to the theory of core models, but we have chosen to avoid this.  
  Let $V = L[\vec E]$ where $\vec E$ is a coherent sequence   of non-overlapping extenders. Then:

\begin{fact} \label{GCH}
    GCH holds.
\end{fact}

\begin{proof} 
      This is immediate from \cite[Theorem 3.24]{Mitchell1}. 
\end{proof}

\begin{fact} \label{3.19} 
    If there is a $(\kappa, \kappa +1 + \beta)$-extender $E$ such that $o^{j_E(\vec E)}(\kappa) = \beta$, then
   $\beta < o^{\vec E}(\kappa)$ and $E = E(\kappa, \beta)$. In particular this will hold whenever
 $j_E(\vec E) \restriction (\kappa, \beta)   = \vec E \restriction (\kappa, \beta)$. 
\end{fact}

\begin{proof}  This is immediate from  \cite[Lemma 3.19]{Mitchell1}. 
\end{proof} 

\begin{fact} \label{3.18} Let $E$ be an extender and $j_E: V \longrightarrow M_E = \Ult(V, E)$ the corresponding
   ultrapower map.  Then comparing $V$ and $M_E$ leads to iterations $i_0: V \longrightarrow N$ and $i_1: M_E \longrightarrow N$
   with a common target model $N$, and $i_1 \circ j_E = i_0$. 
\end{fact} 
   
\begin{proof} A model $L[\vec E]$ is said to be {\em $\phi$-minimal} if $L[\vec E]$ is a model of $\phi$,
   but for no initial segment $\vec E'$ of $\vec E$ is $L[\vec E']$ a model of $\phi$.
 Suppose the claim fails, and let $L[\vec E]$ be $\phi$-minimal for the formula ``there is an extender $E$ such that the
   claim fails''. Appealing to  \cite[Proposition 3.18]{Mitchell1} we obtain an immediate contradiction.
\end{proof}

\begin{fact} \label{jensen} 
 Let $\mu$ be a cardinal with $\kappa^{++} \le \mu < o^{\vec E}(\kappa)$. Then for every $\nu$ with $\mu \le \nu < o^{\vec E}(\kappa)$,
   every bounded subset of $\mu$ appears   in $\Ult(V, E(\kappa, \nu))$.  
\end{fact}

\begin{proof} We start by proving that if $\lambda > \kappa$ and  $o^{\vec E}(\kappa) \ge \lambda^+$, then every subset of $\lambda$
   appears in $\Ult(V, E(\kappa, \beta))$ for some $\beta$ with $\lambda < \beta < \lambda^+$. To see this, let $A \subseteq \lambda$
   and find some large regular $\theta$ such that $A \in L_\theta[\vec E]$. Let $X \prec L_\theta[\vec E]$ be such that
   $\vert X \vert = \lambda$, $P(\kappa) \subseteq X$, $A \in X$ and $X \cap \lambda^+ \in \lambda^+$. Let
   $M$ be the transitive collapse of $X$, so that $A \in M$ and $M = L_{\bar\theta}[\vec F]$ for some extender sequence
   $\vec F$.

   It is clear that $\vec E \restriction \kappa = \vec F \restriction \kappa$, and
   since $P(\kappa) \subseteq M$ a straightforward induction using Fact \ref{3.19} shows that for every
   $\zeta < o^{\vec F}(\kappa)$ we have $\zeta < o^{\vec E}(\kappa)$ and $E(\kappa, \zeta) = F(\kappa, \zeta)$.
   Let $\beta = o^{\vec F}(\kappa)$, and note that $\lambda < \beta < \lambda^+ \le o^{\vec E}(\kappa)$. 
   Compare the models $M$ and $\Ult(V, E(\kappa, \beta))$,  to obtain iterations 
   $i_0: M \longrightarrow N_0$ and $i_1 :  \Ult(V, E(\kappa, \beta)) \longrightarrow N_1$.  By coherence
   and the agreement between $\vec E$ and $\vec F$, together with the non-overlapping condition, we see that
   both $i_0$ and $i_1$ have critical points greater than $\lambda$.

   In the comparison it is not possible that $M$ out-iterates $\Ult(V, E(\kappa, \beta))$, for then
   $M$ would out-iterate $V$ and we could obtain a set of indiscernibles for $V$. It follows that $N_0 \subseteq N_1$.
   Since $A \in M$ and the critical points of $i_0, i_1$ are greater than $\lambda$ we see successively that
   $A \in N_0$, $A \in N_1$ and finally $A \in \Ult(V, E(\kappa, \beta))$.

   To finish the proof of Fact \ref{jensen} let $\kappa$, $\mu$, and $\nu$  be as above, and let $B$ be a bounded
   subset of $\mu$, so that without loss of generality $B \subseteq \lambda$ for some cardinal
   $\lambda < \mu$. It follows from what was proved above that $B \in \Ult(V, E(\kappa, \beta))$ where
   $\lambda < \beta < \lambda^+ \le \mu$. Now $E(\kappa, \beta) \in \Ult(V, E(\kappa, \nu))$
   and the models $V$ and $\Ult(V, E(\kappa, \nu))$ agree past $\kappa$, so that easily
   their ultrapowers by $E(\kappa, \beta)$ agree past the image of $\kappa$, and so
   $B \in \Ult(V, E(\kappa, \nu))$ as claimed. 
\end{proof}

   Using these results, we can characterise the $(\kappa, \mu)$-extenders which witness that
   $\kappa$ is $\mu$-strong. 

\begin{lemma} \label{mainlemma} 
   Let $V = L[\vec E]$  where $\vec E$ is a non-overlapping coherent sequence   of extenders. 
   Let $\kappa < \mu = \vert V_\mu \vert$. Then:  

\begin{enumerate} 

\item  For every $\bar \mu$ such that $\mu \le \bar\mu < o^{\vec E}(\kappa)$, 
   the extender $E(\kappa, \bar\mu) \restriction \mu$ witnesses that
   $\kappa$ is $\mu$-strong. 

\item  \label{mainbit}  If $E$ is a $(\kappa, \mu)$-extender witnessing that $\kappa$ is
     $\mu$-strong, then $E = E(\kappa, \bar\mu) \restriction \mu$ for
    some $\bar\mu$ such that $\mu \le \bar\mu < o^{\vec E}(\kappa)$.

\end{enumerate} 

\end{lemma}

\begin{proof}   The first claim is straightforward. Since $\mu$ is a limit cardinal greater than
    $\kappa$, appealing to Fact  \ref{jensen} we see that  all bounded subsets
    of $\mu$ are in $\Ult(V, E(\kappa, \bar\mu))$, so that $V_\mu \subseteq \Ult(V, E(\kappa, \bar\mu))$.
    The extender $E(\kappa, \bar\mu) \restriction \mu$ is the $(\kappa, \mu)$-extender approximating
    the embedding $j_{E(\kappa, \bar\mu)}$, so it witnesses that $\kappa$ is $\mu$-strong.

    We prove the second claim by contradiction. If it fails, we may assume that $L[\vec E]$ is
    $\phi$-minimal where $\phi$ asserts ``the second claim fails''. Fix an extender $E$ witnessing that $\kappa$ is
    $\mu$-strong, and form $j_E: V \longrightarrow M_E = \Ult(V, E)$. Let $\vec F = j_E(\vec E)$, so
    that $\vec F$ is a coherent non-overlapping sequence of extenders in $M_E = L[\vec F]$. Now
    compare $V$ and $M_E$: by Fact \ref{3.18} above we get iterations $i_0: V \longrightarrow N$ and
    $i_1: M_E \longrightarrow N$ such that $i_1 \circ j_E = i_0$.

    Since $\crit(j_E) = \kappa$, we also have $\crit(i_0) = \kappa$, so that the first extender
    which is used in $i_0$ must be of the form $E(\kappa, \bar\mu)$ for some $\bar\mu < o^{\vec E}(\kappa)$.
    Since $\crit(j_E) = \kappa$, the extender sequences $\vec E$ and $\vec F$ agree up to $\kappa$.
    An easy induction using Fact \ref{3.19} shows that for every $\eta < o^{\vec F}(\kappa)$ we have
    $\eta < o^{\vec E}(\kappa)$ and $F(\kappa, \eta) = E(\kappa, \eta)$. 

    Since $V_\mu \subseteq M_E$, for every $\eta < \min \{ \mu, o^{\vec E}(\kappa) \}$ we have
    $E(\kappa, \eta) \in M_E$, so by another appeal to Fact \ref{3.19} we see that
    $\eta < o^{\vec F}(\kappa)$. Summarising, at the critical point $\kappa$ we have that
\begin{itemize}
\item  The sequence ${\vec F}(\kappa, -)$ is an initial segment of  ${\vec E}(\kappa, -)$. 
\item  The sequences ${\vec F}(\kappa, -)$ and ${\vec E}(\kappa, -)$ agree up to $\mu$.
\end{itemize} 
 In the comparison of $L[\vec E]$ and $L[\vec F]$ an extender with critical point $\kappa$ is applied
  at the first step in the iteration $i_0$ of $L[\vec E]$, so  the only possibility is that 
    $\mu \le \bar \mu  = o^{\vec F}(\kappa) < o^{\vec E}(\kappa)$. 
    
   In the first step of the comparison we applied $E(\kappa, \bar \mu)$ on the $V$-side and did nothing
   on the $M_E$-side, obtaining models which  have identical $\bar\mu$-sequences of extenders at
   critical point $\kappa$. Since we are using non-overlapping extender sequences and $\bar \mu \ge \mu$, it follows that
   in the remainder of the comparison all critical points are greater than $\mu$;
   that is to say 
    $\crit(i_1) > \mu$, and  all critical points in $i_0$ past the first step are
    greater than $\mu$. So now for any  $a \in [\mu]^{ < \omega} $ and $X \subseteq [\kappa]^{\vert a \vert}$,
we see that
\[
    a \in j_E(X) \iff a \in i_0(X) \iff a \in j_{E(\kappa, \bar\mu)}(X),
\]
    where the first equivalence holds because $\crit(i_1) > \mu$ and $i_1 \circ j_E = i_0$, and the
   second equivalence holds because the first step in $i_0$ is to apply $E(\kappa, \bar \mu)$ and
   all subsequent critical points in $i_0$ are greater than $\mu$.

   It follows that $E = E(\kappa, \bar \mu) \restriction \mu$. We have shown that the second claim holds in $L[\vec E]$,
   an immediate contradiction.
\end{proof} 
    
The following corollary is immediate:
\begin{corollary} \label{useful} Let $V = L[\vec E]$  where $\vec E$ is a non-overlapping coherent sequence   of extenders.
   Let $o^{\vec E}(\kappa) = \mu+1$ for some cardinal $\mu > \kappa$ with $\vert V_\mu \vert = \mu$. Then
   $E(\kappa, \mu)$ is the only $(\kappa, \mu)$-extender witnessing that $\kappa$ is $\mu$-strong.
\end{corollary} 

\begin{remark}
  The analysis in the proof of part \ref{mainbit} of Lemma \ref{mainlemma} can easily be pushed further, so
    show that the comparison terminates after one step on the $V$-side with
\[  
    N= \Ult(V, E) = \Ult(V, E(\kappa, \bar \mu)).
\] 
\end{remark}

   With Corollary \ref{useful} in hand, we can now prove the remaining results about Question \ref{quest}.

\begin{theorem} \label{badlewit}
 It is consistent that $\mu$ is singular, there is exactly one $(\kappa, \mu)$-extender $E$
     witnessing that $\kappa$ is $\mu$-strong, and $E$ is a bad witness.
\end{theorem}

\begin{proof} The proof is very similar to that of Fact \ref{bad}.
 Suppose that $V = L[\vec E]$ for a non-overlapping
   coherent extender sequence $\vec E$, 
   and $o^{\vec E}(\kappa) = \nu +1$ for some inaccessible cardinal $\nu > \kappa$. Let $E = E(\kappa, \nu)$.
   By the arguments in the proof of Fact \ref{bad} we can find a singular cardinal $\mu$ such that  $\kappa < \mu < \nu$,
   and $j_E$ factors as $k \circ j_{E \restriction \mu}$ where $\crit(k) = \mu$ and $k(\mu) = \nu$.
   As we already argued, if we let $F = E \restriction \mu$ then $F$ is a bad witness.

   The novel point is that
   since $k$ is elementary and  $o^{j_E(\vec E)}(\kappa) = \nu$, we have
   $o^{j_F(\vec E)}(\kappa) = \mu$, so that $E(\kappa, \mu) = F$ by an appeal to Fact \ref{3.19} above.
   Now we take the ultrapower of $V$ by $E(\kappa, \mu+1)$, and obtain (by coherence and Corollary \ref{useful}) a model $N$ in which
   $F$ is still a bad witness and $F$ is the unique $(\kappa, \mu)$-extender witnessing that
   $\kappa$ is $\mu$-strong.  
\end{proof} 

\begin{theorem} 
 It is consistent that $\mu$ is singular, there is exactly one $(\kappa, \mu)$-extender $E$
     witnessing that $\kappa$ is $\mu$-strong, and $E$ is a good witness.
\end{theorem}

\begin{proof}
 Suppose that $V = L[\vec E]$ for a non-overlapping   coherent extender sequence $\vec E$, 
 and that $\kappa$ is minimal such that there exists $\mu > \kappa$ with $o^{\vec E}(\kappa) = \mu +1$
  and $\mu = \vert V_\mu \vert$. We claim that $\mu$ is singular. Otherwise $\mu$ is inaccessible,
  so there is $\bar \mu < \mu$ such that $\bar \mu = V_{\bar \mu}$. Now if $E =  E(\kappa, \bar \mu)$
  then $E$ witnesses ``$\kappa$ is $\bar \mu$-strong'' and 
\[
   Ult(V, E) \models \mbox{``$\bar \mu = V_{\bar \mu}$ and $o^{j_E(\vec E)}(\kappa) = \bar \mu$''}, 
\]  
   so that by elementarity we obtain a contradiction to the minimal choice of $\kappa$. 

  By Corollary \ref{useful}, if we let $F =E(\kappa, \mu)$ then $F$ is the unique
 $(\kappa, \mu)$-extender witnessing that $\kappa$ is $\mu$-strong. 
  We claim that $F$ is a good witness. If not then $\mu$ is inaccessible in $\Ult(V, F)$, 
  but since $\mu = o^{j_F(\vec E)}(\kappa)$ this would imply by the elementarity of $j_F$  that there are many ordinals
  $\delta < \kappa$ such that $o^{\vec E}(\delta)$ is inaccessible, contradicting the minimal
  choice of $\kappa$. 
\end{proof}  

\section{Conclusion and open questions} 

   We have determined fairly tight upper and lower bounds in consistency strength for the existence 
   of a bad witness, and have produced models in which the unique witness is good or bad as we please. 
   The following questions are open:

\begin{enumerate}

\item  Determine the exact consistency strength of the existence of a bad witness. We note that (by a straightforward argument)
     if there is a bad witness there is a model of the form $L[\vec E]$ with a bad witness, so the question amounts to
     asking how long the sequence $\vec E$ must be before a bad witness appears. 

\item Is it consistent that $\mu$ is singular and there are exactly two $(\kappa, \mu)$-extenders witnessing
     ``$\kappa$ is $\mu$-strong'', of which one is good and the other is bad?

\end{enumerate} 

\section*{Acknowledgements} 

  The authors would like to thank Brent Cody for raising Question \ref{quest}, for some interesting discussions regarding
  this question, and for pointing out that Fact \ref{bad} (which we had also observed) was already known and had appeared
  in reference \cite{FrHon}.   
 An instance of Theorem \ref{goodwit} appears in Cody's thesis \cite{Cody}.

\end{document}